\documentclass[letterpaper, 10 pt, conference]{ieeeconf}  

\IEEEoverridecommandlockouts                              
\usepackage{tikz}
\usepackage{subcaption}
\usepackage{booktabs}
\usepackage{url}
\usepackage{cite}
\usepackage{amsmath,amssymb,amsfonts}
\usepackage{algorithmic}
\usepackage{graphicx}
\usepackage{textcomp}
\usepackage{xcolor}
\def\BibTeX{{\rm B\kern-.05em{\sc i\kern-.025em b}\kern-.08em
    T\kern-.1667em\lower.7ex\hbox{E}\kern-.125emX}}
\newtheorem{thm}{Theorem }
\newtheorem{prop}{Proposition}
\newtheorem{lem}{Lemma }
\newtheorem{defn}{Definition}
\newtheorem{rem}{Remark }
\newtheorem{ass}{Assumption}
\newtheorem{ex}{Example}

\usepackage[ruled,vlined,linesnumbered]{algorithm2e}
\usepackage{makecell}
\usepackage{dsfont}

\newcommand{\Real}{\mathbb R}
\newcommand{\B}{\mathbb B}
\newcommand{\eps}{\varepsilon}

\newcommand{\norm}[1]{\left\Vert#1\right\Vert}
\newcommand{\ra}{\rightarrow}
\newcommand{\set}[1]{\left\{#1\right\}}

\newcommand{\D}{\mathcal{D}}

\newcommand{\I}{\mathcal{I}}
\newcommand{\R}{\mathcal{R}}

\let\subset\subseteq

\renewcommand{\S}{\mathcal{S}}
\renewcommand{\subset}{\subseteq}

\newcommand{\Sd}{\mathcal{S}_\delta}

\newcommand{\Ds}{\mathcal{D}_{\operatorname{s}}}
\newcommand{\Dd}{\mathcal{D}_{\operatorname{0}}}

\newcommand{\dx}{\dot{x}}
\usepackage{xcolor}

\newtheorem{eg}[thm]{\bf Example}
\newcommand{\qed}{\hfill $\square$}

\title{\LARGE \bf
Towards Learning and Verifying Maximal Lyapunov-Barrier Functions with a  Zubov PDE Formulation}

\author{Yiming Meng and Jun Liu 
\thanks{This research was supported in part by an NSERC Discover Grant and the Canada Research Chairs program. The authors also acknowledge the Mathematics Faculty Computing Facility (MFCF) at the University of Waterloo for computing support.}
\thanks{Both authors are with the Department of Applied Mathematics, Faculty of Mathematics, University of Waterloo, Waterloo, Ontario N2L 3G1, Canada.  Emails: \texttt{yiming.meng@uwaterloo.ca, j.liu@uwaterloo.ca}
        }
}

\begin{document}

\maketitle
\thispagestyle{empty}
\pagestyle{empty}

\begin{abstract}
Verifying stability and safety guarantees for nonlinear systems has received considerable attention in recent years. This property serves as a fundamental building block for specifying more complex  system behaviors and   control objectives. However, estimating the domain of attraction under safety constraints and constructing a Lyapunov–barrier function remain challenging tasks for nonlinear systems. To address this problem, we propose a Zubov PDE formulation with a Dirichlet boundary condition for autonomous nonlinear systems and show that a physics-informed neural network (PINN) solution, once formally verified, can serve as a Lyapunov–barrier function that jointly certifies stability and safety. This approach extends existing converse Lyapunov–barrier theorems by introducing a PDE-based framework with boundary conditions defined on the safe set, yielding a near-optimal certified under-approximation of the true safe domain of attraction.
\end{abstract}
\begin{keywords}
Lyapunov-barrier function, Dirichlet-Zubov, PDE characterization, learning,  physics-informed neural networks, formal verification. 
\end{keywords}

\section{Introduction}

Stability-with-safety specifications serve as fundamental building blocks for defining more complex system behaviors. 
Verification and control synthesis of such specifications have attracted considerable attention in a wide range of safety-critical applications. 
Lyapunov-based methods, including Lyapunov functions (LFs), barrier functions (BFs), and their control counterparts (CLFs and CBFs), enable verification and design of controllers ensuring stability and set invariance without requiring state-space discretization\cite{ames2016control, hsu2015control, nguyen2020dynamic}. These Lyapunov-based approaches have gained renewed interest due to their elegant mathematical foundation and conceptual simplicity, coupled with tractable verification conditions \cite{dai2022convex, ahmadi2016some,  topcu2008local}. 

However, a remaining challenge is that LFs and BFs cannot be easily combined to simultaneously ensure both stability and safety for autonomous and control systems. In particular, for control systems, the controller is often forced to prioritize one objective (safety or stability) at the expense of the other. Recent works have aimed to learn a unified Lyapunov–barrier function \cite{dawson2022safe, romdlony2016stabilization}; however, this approach becomes difficult to apply to complex nonlinear dynamics, especially when the invariant set is unknown, and obtaining verifiable guarantees remains a challenge. 

 
Other related work includes the formal synthesis and verification of neural network Lyapunov functions \cite{liu2024lyznet} and the simultaneous learning of formally verified stabilizing controllers \cite{chang2019neural, zhou2022neural}. Recent research has also explored learning local Lyapunov functions \cite{grune2021computing, gaby2022lyapunov}, control Lyapunov functions (CLFs) \cite{rego2022learning, grune2023examples}, and control Lyapunov–barrier functions (CLBFs) \cite{dawson2022safe} using deep neural networks for potentially high-dimensional systems, typically relying on training with a Lyapunov-inequality-based loss. However, the resulting Lyapunov functions are generally not formally verified. In \cite{meng2024physics}, PINNs were employed to solve Hamilton–Jacobi–Bellman (HJB) equations for optimal control, but the approach relied on an initial stabilizing controller and yielded only local solutions. The work in \cite{liu2025formally} proposed a transformed HJB formulation, demonstrating that verified neural CLFs outperform sum-of-squares and rational CLFs in estimating the null-controllability set of nonlinear systems. However, state constraints were not explicitly considered. The study in \cite{li2025solving} introduced a discounted stabilize–avoid value function, which provides a set-level safety certificate but does not constitute a Lyapunov–barrier feedback law guaranteeing finite-time reachability and safety. The   work \cite{serry2025safe} characterized near-maximal safe domains of attraction, although it focuses on discrete-time systems. The work in \cite{grune2015zubov} employs a PDE formulation but relies on a nontrivial terminal cost, which leads to less convenient boundary conditions and a corresponding Hamilton–Jacobi equation that cannot be easily solved using existing computational tools \cite{liu2024lyznet}. Moreover, the resulting solution does not directly qualify as a Lyapunov–barrier function by definition.

To overcome the limitations of existing Lyapunov–barrier function approaches, we draw inspiration from recent advances in solving the Zubov equation using PINNs \cite{liu2025physics}, which enable the computation of a maximal Lyapunov function whose domain theoretically coincides with the true domain of attraction. In this work, we focus on autonomous systems and propose a PDE-based Lyapunov–barrier formulation that jointly certifies both stability and safety. Our approach generalizes existing converse Lyapunov–barrier theorems \cite{meng2022smooth,romdlony2016stabilization} by introducing a PDE framework with boundary conditions prescribed on the safe set, thereby yielding a formally certified, near-optimal under-approximation of the true safe domain and region of attraction. This framework is expected to lay the foundation for future studies on PDE-based construction of control Lyapunov–barrier functions.

 
\textbf{Notation:} We list some notation used in this paper. 
For $x\in\Real^n$ and $r\ge 0$, we denote the ball of radius $r$ centered at $x$ by $\B(x;r)=\set{y\in\Real^n:\,\norm{y-x}\le r}$, where $\norm{\cdot}$ is the Euclidean norm. For a closed set $A\subset\Real^n$ and $x\in\Real^n$, we denote the distance from $x$ to $A$ by $\norm{x}_{A}=\inf_{y\in A}\norm{x-y}$. 
For a set $A\subseteq\Real^n$, $\overline{A}$ denotes its closure.  
For two sets $A,B\in\Real^n$, we use $A\setminus B$ to denote the set difference defined by $A\setminus B=\set{x:\,x\in A,\,x\not\in B}$.  We also write $a\gtrsim b$ if there exists a $C>0$ (independent of $a$ and $b$) such that $a\geq Cb$.

\section{Preliminaries}\label{sec:prelim}
\subsection{Dynamical Systems and Concept of Stability-with-Safety}
Consider a continuous-time dynamical system 
\begin{equation}\label{eq:sys}
     \dx= f(x),
\end{equation}
where $x\in\Real^n$ and $f:\,\Real^n\ra\Real^n$ is assumed to be locally Lipschitz. For each $x_0\in\Real^n$, we denote the unique solution starting from $x_0$  by $\phi(t;x_0)$ for all $t\in\I$, where $\I$ is the maximal interval of existence. We assume $\I = [0, \infty)$ throughout the paper.  

For $T\ge 0$, we define $\R^{\geq T}(x_0)=\set{\phi(t;x_0): t\geq T}$ 
and  write $\R(x_0) = \R^{\geq 0}(x_0)$. 
For a set $\mathcal{X}_0\subseteq\Real^n$, we define 
$\R(\mathcal{X}_0)   =  \bigcup_{x_0\in \mathcal{X}_0} \R(x_0)$.

\begin{defn}[Forward invariance]
A set $\Omega\subset\Real^n$ is said to be forward invariant for \eqref{eq:sys}, if solutions from $\Omega$ satisfy $\mathcal{R}(\Omega)\subseteq\Omega$. \qed
\end{defn}

Let $x_{\operatorname{eq}}$ be a locally asymptotically stable equilibrium point  of \eqref{eq:sys}. Without loss of generality, we assume that 
$x_{\operatorname{eq}}=0$ throughout the remainder of the paper.
\begin{defn}[Domain of attraction]
 The domain of attraction of the origin is defined as  
\begin{equation*}
    \Dd  = \big\{x\in \Real^n:\, \lim_{t\ra\infty}\norm{\phi(t;x)} =0\big\}. \tag*{\(\square\)}
\end{equation*}
\end{defn}

\begin{defn}[Stability-with-safety guarantee]
Denoting $\mathcal{X}_0\subseteq\Real^n$ as the set of initial conditions and $U\subseteq\Real^n$ as the obstacle,
we say that \eqref{eq:sys} satisfies a stability-with-safety   specification $(\mathcal{X}_0,U)$   if the following conditions hold:
\begin{enumerate}
    \item (asymptotic stability) the origin is asymptotically stable  and   $\mathcal{X}_0\subseteq\Dd$. 
    \item (safe w.r.t. $U$) $\R(\mathcal{X}_0)\cap U=\emptyset$. \qed
\end{enumerate} 
\end{defn}

\subsection{Converse Lyapunov–barrier Function Theorem from Previous Work}
Before proceeding, we introduce the following concept of a proper indicator  defined on an open domain \cite{teel2000smooth}, which generalizes the idea of a radially unbounded function on $\Real^n$. 

\begin{defn}
Let $\{0\}$ be contained in an open set $D\subset\Real^n$. A continuous function $\omega:\,D\ra\Real_{\ge 0}$ is said to be a proper indicator for the origin on $D$ if the following two conditions hold: (1) $\omega(x)=0$ if and only if $x = 0$; (2) $\lim_{m\ra\infty} \omega(x_m)=\infty$ for any sequence $\set{x_m}$ in $D$ such that either $x_m\ra p\in \partial D$ or $\norm{x_m}\ra\infty$ as $m\ra\infty$. 
\end{defn} 

We exemplify the construction of a proper indicator as follows.

\begin{ex}\label{eg: indicator_2}
     Suppose $D = \set{x\in\Real^n: \gamma(x)>0}$ and $\partial D = \set{x\in\Real^n: \gamma(x)=0}$, with $\gamma\in\mathcal{C}^\infty(\Real^n)$ such that $\nabla \gamma(x)\neq 0$ for all $x\in\partial D$. 
    Suppose that $0\in D$, and $\overline{D}$ is bounded. Then, by continuity of $\gamma$, $\overline{D}$ is closed and hence compact. Therefore,  the image of $\overline{D}$ is compact, which implies that $\gamma$ attains both maximum and minimum on $\overline{D}$. It is clear that $\min_{x\in\overline{D}}\gamma(x) = 0$ and $p:=\max_{x\in\overline{D}}\gamma(x) = \max_{x\in D}\gamma(x)>0$ (due to  $D\neq \emptyset$). 

    Now, consider 
    \begin{equation}\label{E: omega_1}
        \omega_1(x) = \theta(x)\cdot\frac{\norm{x}^2}{\gamma^k(x)}, \;k\geq 1,
    \end{equation}
    and $\omega_2(x) =  \norm{x}^2 \cdot (1-\log(\gamma(x)/p)),$
where $\theta\in\mathcal{C}_b^1(\Omega; \Real_+)$ is a positive, bounded scaling function defined on $\Omega\supseteq D$. 
    By definition, one can verify that both of the above constructions are proper indicators on $D$. \qed

\end{ex}

The work in \cite{meng2022smooth} considered uniting Lyapunov and barrier functions   based on the concept of a proper indicator. Specifically, if the domain of a Lyapunov function is restricted to the set of initial conditions from which trajectories can simultaneously satisfy the conditions of asymptotic stability and safety, then a smooth Lyapunov function can be constructed that satisfies the following theorem, which is rephrased to be consistent with the notation used in this paper. 

\begin{thm}\label{thm:conv1}
Suppose that  $U$ is closed  and $0\notin U$. Then the following two statements are equivalent: 
\begin{enumerate}
    \item System \eqref{eq:sys} satisfies the stability with safety guarantee specification  $(\mathcal{X}_0,U)$.
    \item There exists an open set $D$ such that $\set{0}\cup \mathcal{X}_0\subset D$ and $D\cap U=\emptyset$, a smooth function $V:\,D\ra\Real_{\ge 0}$ and class $\mathcal{K}_{\infty}$ functions $\alpha_1$, $\alpha_2$, and $\alpha_3$, such that \begin{equation}\label{eq:lyap1}
    \alpha_1(\omega_0(x)) \le V(x) \le \alpha_2(\omega_0(x))
\end{equation}
and
\begin{equation}\label{eq:lyap2}
    \nabla V(x)\cdot f(x)  \le -\alpha_3(\omega_0(x)) 
\end{equation}
for all $x\in D$, {where $\omega_0$ be any proper indicator   on $D$}. 
\end{enumerate}
Moreover, the set $D$ can be taken as the following set
     \begin{align}
& \D\\ =& \big\{x\in \Real^n:\,   \lim_{t\ra\infty}\norm{\phi(t;x)}=0 \text{ and }  \notag 
  \phi(t;x)\not\in U,\forall t\ge 0\big\}.  
\end{align} \qed
\end{thm}

Note that $\D$ has   been  proved to be open, forward invariant, and $\mathcal{X}_0\subset\D\subseteq \Dd$   if \eqref{eq:sys} satisfies a stability-with-safety guarantee specification $(\mathcal{X}_0, U)$ \cite{meng2022smooth}. While the above 
$V$ already serves as both a Lyapunov function and, partially, as a barrier function ensuring safety and stability, it is only defined on a restricted forward-invariant set and therefore cannot fulfill the full role of a Lyapunov–barrier function. That said, the existence of $V$ provides the foundation for constructing a barrier function $B:\Real^n\ra\Real$ using a partition-of-unity type patching approach based on $\sup_{x\in\mathcal{X}_0}V(x)-V(x)$ and its suitable extensions \cite[Proposition 12]{meng2022smooth}, under the assumption that $\mathcal{X}_0$ is compact,, such that
\begin{enumerate}
     \item $\mathcal{X}_0\subset\{x\in\Real^n: B(x)\geq 0\}$ and $B(x)<0$ for all $x\in U$;
    \item $\nabla B(x)\cdot f(x)\geq 0$ for all $x\in\D$. 
\end{enumerate} 
The pair $(V,B)$ is also referred to as a Lyapunov–barrier function pair. 

It is worth noting that the above converse Lyapunov–barrier function theorem is proved analytically. 
In general, finding suitable Lyapunov-like functions is difficult, even when the intrinsic stability or safety properties of a nonlinear system are known. Therefore, this paper employs PDE-based characterizations that yield a Lyapunov candidate over the entire domain $\D$, which can be extended to form a Lyapunov–barrier function pair, and a single Lyapunov–barrier function that   ensures stability and safety.

\subsection{Problem Formulation}
Motivated by the drawbacks of previous work on the converse Lyapunov–barrier function theorem, we aim to: (i) Characterize a Lyapunov function $V$ defined on the entire domain $\D$ via a Lyapunov-like equation formulation; (ii)  Construct a Lyapunov–barrier function $W$ such that $W$ is positive definite and   satisfies the following conditions 
\begin{enumerate}
    \item $\nabla W(x)\cdot f(x)<0$ for all $x\in\D\setminus\{0\}$;
    \item $W(x)<1$ for all $x\in\D$; 
    \item $W(x)\geq 1$ for all $x\in U$,
\end{enumerate}
via a Dirichlet boundary value problem for a Zubov-type equation; (iii)  Predict the domain $\D$ and use a physics-informed neural network (PINN) to solve for $W$; (iv)  Verify the condition  1) for $W$.

\section{Converse Lyapunov-Barrier Function for Stability with Safety Guarantees} \label{sec:stability}

In this section, we  derive a set of PDE-based converse Lyapunov–barrier function theorems for \eqref{eq:sys}, which satisfies a stability-with-safety guarantee specification $(\mathcal{X}_0, U)$ for some unsafe set $U\subseteq\Real^n$. To assist problem solving, we make the following assumptions on $\Ds = U^c$. 
\begin{ass}\label{ass: h}
    We assume 
    \begin{enumerate}
        \item $\Ds = \set{x\in\Real^n: h(x)<1}$ with some $h\in\mathcal{C}^1(\Real^n)$ such that $\nabla h(x)  \neq 0$ on $\set{x\in\Real^n: h(x) = 1}$;
        \item $\Ds$ is open and bounded;
        \item $0\in\Ds$. \qed
    \end{enumerate}
\end{ass}

\begin{rem}
    Consider $W_\delta = \set{x\in\Ds: \norm{x}_{\partial \Ds}<\delta}$. Set $\gamma((x) = 1-h(x)$. Then, there exist $a, b>0$ such that $a\leq \norm{\nabla \gamma(x)}\leq b$ for all $x\in W_\delta$. For $x = y + \norm{x}_{\partial \Ds}\frac{\nabla \gamma((y)}{\norm{\nabla \gamma((y)}}$, we have $\gamma((x)-\gamma((y) = \int_0^1\nabla \gamma((y+\theta(x-y))\cdot (x-y)d\theta$ and $\gamma((x)\leq b\cdot \norm{x}_{\partial \Ds}$. Similarly, we can verify that  $\gamma((x)\geq a\cdot \norm{x}_{\partial \Ds}$, indicating that $\gamma(x)$ is an equivalent metric w.r.t. $\norm{x}_{\partial \Ds}$ in the boundary layer $W_\delta$. \qed
\end{rem}

\subsection{Maximal Converse Lyapunov Function for Stability-with-Safety Specification}

This subsection employs ideas from the Lyapunov equation to explicitly construct a maximal Lyapunov function for $(\mathcal{X}_0, U)$ via a proper indicator defined on $\Ds = U^c$, and establishes its properties.

To proceed, we define the exit time from $\Ds$, or equivalently entry time in $U$, as
\begin{equation}
    \tau(x) = \inf\set{t\geq 0: \phi(t;x)\in U}. 
\end{equation}
We use the convention that $\inf\emptyset = +\infty$, so that $\tau(x) = +\infty$ whenever the trajectory never reaches $U$.  It is clear that $\D\subseteq\Ds$  and 
\begin{equation}
\begin{split}
         \D    = & \big\{x\in \Real^n:\,   \lim_{t\ra\infty}\norm{\phi(t;x)}=0 \text{ and }  \notag 
  \tau(x) = +\infty\big\}\\
   = & \Dd\cap \big\{x\in \Real^n:\,    
  \tau(x) = +\infty\big\}. 
\end{split}
\end{equation} For future reference, the following \textit{Dynamic Programming Principle} for $\tau(x)$ is provided.

\begin{prop}\label{prop: tau_DPP}
    Let $\overline{\Real} = \Real\cup\set{-\infty, \infty}$ denote the extended real number system, equipped with the standard extended real arithmetic  $a + (+\infty) = +\infty$ for all $a\in\Real\cup\set{+\infty}$. Then, for all $x\in \Real^n$, 
    \begin{equation}\label{E: tau_DPP}
        \tau(x) = t\wedge\tau(x) + \mathds{1}_{\set{t\leq \tau(x)}}\tau(\phi(t;x)), \;\;t\geq 0
    \end{equation}
Furthermore, for all $x\in\Real^n$ and for all $t\geq 0$, $\tau(x)<\infty$ if and only if $\mathds{1}_{\set{t\leq \tau(x)}}\tau(\phi(t;x))<\infty$. \qed\end{prop}
\begin{proof}
    For $\tau(x)<+\infty$, the first part of the proof falls in the same procedure as \cite[Chap. Iv, Prop. 2.1]{bardi1997optimal}. For $\tau(x)=+\infty$, the conclusion of \eqref{E: tau_DPP}  follows by transitivity of reachability. 

For the second part, if $\tau(x)<\infty$, then either $t\leq \tau(x)$, which implies, by \eqref{E: tau_DPP},  that $\mathds{1}_{\set{t\leq \tau(x)}}\tau(\phi(t;x)) = \tau(x)-t<\infty$; or $t>\tau(x)$, implying $\mathds{1}_{\set{t\leq \tau(x)}}\tau(\phi(t;x)) = 0<\infty$.  Conversely, 
    if $\tau(x)=\infty$, then $\set{t\geq 0: \phi(t;x)\in U} = \emptyset$, and $\mathds{1}_{\set{t\leq \tau(x)}}\tau(\phi(t;x)) = \tau(\phi(t;x)) = \infty$. 
\end{proof}

We then construct $V$  via a proper indicator $\omega(x)$ defined on $\Ds$. For $\omega: \Ds\ra \Real$, define 
\begin{equation}\label{E: V}
V(x) =
\begin{cases}
\displaystyle \int_0^{\tau(x)} \omega(\phi(t;x))\,dt + q(\phi(\tau(x);x)), & \tau(x) < \infty, \\[1.2ex]
\displaystyle \int_0^{\infty} \omega(\phi(t;x))\,dt, & \tau(x) = \infty.
\end{cases}
\end{equation}
where $g$ is the terminal cost satisfying $q(x) = \infty$ for $x\in\partial\Ds\cup(\Real^n\setminus\Ds)$ and $q(x) = q_0(x)$ for $x\in\Ds$, where   $g_0\in\mathcal{C}(\Ds)$ and $g_0(x)  \ra +\infty$ as $x\ra\partial\Ds$. The following is the \emph{Dynamic Programming Principle} for $V$. 

\begin{lem}
    Consider $V$ defined in \eqref{E: V}. Then, 
\begin{equation}\label{E: V_DPP2}
     V(x)  =   \int_0^{t\wedge \tau(x)}\omega(\phi(s;x))ds +  V(\phi(t\wedge \tau(x);x))
\end{equation}
for all $x\in \Real^n$ and for all $t\geq 0$. \qed
\end{lem}

\begin{proof}
For any $x\in\Real^n$, let $y = y(t,x): = \phi(t\wedge\tau(x);x)$. We then have $\tau(y) = \mathds{1}_{\set{t\leq \tau(x)}}\tau(\phi(t;x))$. Then,  by \eqref{E: V},
\begin{equation}\label{E: P1P2}
\begin{split}
V(y) &= 
  \underbrace{\int_0^{\tau(y)} \omega(\phi(s;y))\,ds}_{P_1} 
  + 
  \underbrace{q(\phi(\tau(y);y))\mathds{1}_{\{\tau(y)<\infty\}}\vphantom{\int_0^{\tau(y)}}}_{P_2},
\end{split}
\end{equation}
where we slightly abuse the notation by letting $P_2 = q(\phi(\tau(y);y))$ if $\mathds{1}_{\{\tau(y)<\infty\}} =1$, and $P_2 = 0$ otherwise. 

However, with the change of variables $s = u-(t\wedge \tau(x))$ and the property in \eqref{E: tau_DPP}, 
\begin{equation}\label{E: P1}
    \begin{split}
        P_1 & = \int_0^{\tau(y)}\omega(\phi(s; \phi(t\wedge \tau(x);x)))ds \\
        & = \int_{t\wedge\tau(x)}^{t\wedge\tau(x) + \tau(y)}\omega(\phi(u;x))du\\
        & = \int_{t\wedge\tau(x)}^{\tau(x)}\omega(\phi(u;x))du. 
    \end{split}
\end{equation}
On the other hand, by Proposition \ref{prop: tau_DPP}, we have $\mathds{1}_{\{\tau(y)<\infty\}} = \mathds{1}_{\{\tau(x)<\infty\}}$. It follows that
\begin{equation}\label{E: P2}
    \begin{split}
  P_2=    &   q(\phi(\tau(y);y))\mathds{1}_{\{\tau(y)<\infty\}} \\
      = & q(\phi(\mathds{1}_{\set{t\leq \tau(x)}}\tau(\phi(t;x)); \phi(t\wedge \tau(x);x)))\mathds{1}_{\set{\tau(x)<\infty}}\\
      =  & q(\phi(\mathds{1}_{\set{t\leq \tau(x)}}\tau(\phi(t;x)) + t\wedge \tau(x);x))\mathds{1}_{\set{\tau(x)<\infty}}\\
      = & q(\phi(\tau(x);x))\mathds{1}_{\set{\tau(x)<\infty}}, 
    \end{split}
\end{equation}
Therefore, by \eqref{E: V} and combining   \eqref{E: P1P2}, \eqref{E: P1}, and \eqref{E: P2},
\begin{equation*}
    \begin{split}
        V(x)  & = \int_0^{\tau(x)}\omega(\phi(t;x))dt + q(\phi(\tau(x);x))\mathds{1}_{\set{\tau(x)<\infty}}\\
        & = \int_0^{t\wedge\tau(x)}\omega(\phi(t;x))dt + P_1 + P_2  \\
        & = \int_0^{t\wedge\tau(x)}\omega(\phi(t;x))dt + V(\phi(t\wedge\tau(x);x)),
    \end{split}
\end{equation*}
which completes the proof. 
\end{proof}


To verify whether \eqref{E: V} provides a valid construction of a Lyapunov function on $\D$, we impose the following additional  conditions concerning the interplay between the choice of $\omega$ and the integral form of $
V$ in \eqref{E: V}.

\begin{ass}\label{ass: ass_1}
 Suppose $\Dd\setminus\{0\}\neq \emptyset$. Given a proper indicator $\omega:\Ds\ra\Real$ for $\set{0}$ and $\Ds$, we also assume that
    \begin{enumerate}
    \item For any $\delta>0$, there exists $c>0$ such that $\omega(x)>c$ for all $x$ such that $\norm{x}>\delta$. 
    \item There exists an $r$-neighborhood of $\partial\Ds$ and a   function 
$\psi:\Real_+\ra\Real_+$, such that $\psi$ is non-increasing, $\psi(r)\ra \infty$ as $r\ra 0$,  $\omega(x)\gtrsim\psi(\norm{x}_{\partial\Ds})$, and $\int_0^r\psi(x)dx = \infty$. 
        \item  There exists some $\rho>0$ such that $V$ defined in \eqref{E: V} converges for all $x$ such that $\norm{x}<\rho$;
        \item  For any $\eps>0$, there exists $\delta>0$ such that $\norm{x}<\delta$ implies $V(x)<\eps$. \qed
    \end{enumerate}
\end{ass}

Note that items 1), 3), and 4) of Assumption \ref{ass: ass_1} are the same as those in Assumption 1 of \cite{liu2025physics}, except for the different choice of $\omega$, where the $\omega$ used in \cite{liu2025physics} can be regarded as a proper indicator for $\set{0}$ and $\Real^n$. Item 2) of Assumption \ref{ass: ass_1} is intended to make the integral form of 
$V$ diverge in a neighborhood of $\partial\Ds$, where the integrand also diverges, so that when the trajectory reaches the boundary of $\Ds$, the value of 
$V$ remains compatible with the terminal cost $q$. Under these assumptions, $V$ is a maximal Lyapunov function on 
$\D$ in the sense that $\D = \set{x\in\Real^n: V(x)<\infty}$, with the following properties.

\begin{thm}
    The function $V:\Real^n\ra\Real\cup\set{+\infty}$ defined by \eqref{E: V} satisfies the following properties:
    \begin{enumerate}
        \item $V(x)<\infty$ if and only if $x\in\D$;
        \item $V(x)\ra\infty$ as $x\ra\partial\D$;
        \item $V$ is positive definite on $\D$;
        \item $V$ is continuous on $\D$ and its right-hand derivative along the solution of
        \eqref{eq:sys} satisfies
        $\dot{V} (x): = \lim_{t\ra 0^+}\frac{V(\phi(t;x))-V(x)}{t}= -\omega(x)$ for all $x\in\D$. 
    \end{enumerate}
\end{thm}
\begin{proof} (1) Suppose $x\in\D$. Then $\tau(x) = \infty$, and by \eqref{E: V_DPP2}, $V(x) = \int_0^t\omega(\phi(s;x))ds +V(\phi(t;x))$ for all $t\geq 0$. The finiteness of $V$ in this cases can be proved in the same way as in \cite[Proposition 1-1)]{liu2023towards}. Now suppose $x\notin\D$. Then either $\tau(x)<\infty$, which directly implies $V(x)=\infty$, or $x\in \Dd^c\cap\{x\in\Real^n:\tau(x)=\infty\}$. For the latter case, since $\Dd$ is open, there exists $\delta > 0$ such that $B(0;\delta) \subset \Dd$. Consequently, $\phi(t;x) \notin B(0;\delta)$ for all $t \ge 0$, and by Assumption \ref{ass: ass_1}-(1),  $V(x)=\infty$.  

(2) Denote $S = \{x\in\Real^n: \tau(x) = \infty\}$.Then $\partial\D = (\partial\Dd\cap \overline{S})\cup(\partial S\cap \overline{\Dd})$. For any $y\in\partial \Dd$, the proof proceeds in the same way as in \cite[Proposition 1-2)]{liu2023towards}. It suffices to show the case where   $y\in\partial S\cap \Dd\subseteq\partial\D$. Note that $\tau(y)<\infty$ in this case, since otherwise $y\in\D$ by definition, which cannot be true because $\D$ is open \cite{meng2022smooth}.

Define $T_\delta(x) = \inf\set{t\geq 0: \|\phi(t;x)\|<\delta}$ for any $\delta>0$. Due to the asymptotic stability of the origin, there exists a local Lyapunov function $v$ such that $\dot{v}<0$ in some neighborhood of $0$, and, for all sufficiently small $\mu>0$,  the sublevel set $\set{x: v(x)\leq \mu}$ is forward invariant and satisfies $\set{x: v(x)\leq \mu}\cap U = \emptyset$. We choose $\delta$ sufficiently small such  that $\B(0;\delta)\subset\set{x: v(x)\leq \mu}$ and $\delta$ also satisfies condition (4) of Assumption~\ref{ass: ass_1}. 
Consider $\set{x_m}\subseteq\D$ and $x_m\ra y$. Since $T_\delta(x)<\infty$ for all $x\in\Dd$, the continuity of $T_\delta$ within $\Dd$ also holds. Therefore, $T_\delta(x_m)\ra T_\delta(y)<\infty$, and $\tau(\phi(T_\delta(y);y))=\infty$ as $\phi(T_\delta(y);y)$ lies in  a forward invariant set that   does not intersect $U$. The latter also implies that $\tau(y)< T_\delta(y)$. Indeed, if the opposite holds, i.e. $\tau(y)>T_\delta(y)$ (the equality cannot occur due to the empty intersection of the two corresponding sets),    we would have $\phi(t; \phi(T_\delta(y);y))\notin U$ for all $t\geq 0$,  which indicates $\tau(y)=\infty$ and renders a contradiction.  

Now, by \eqref{E: V_DPP2}, 
\begin{equation*}
    \begin{split}
        V(x_m) & = \int_0^{T_\delta(x_m)}\omega(\phi(s;x))ds + V(\phi(T_\delta(x_m);x_m)),
    \end{split}
\end{equation*}
where $V(\phi(T_\delta(x_m);x_m))<\eps$  for all $m$ by (4) of Assumption \eqref{ass: ass_1}. Since $T_\delta(x_m)\ra T_\delta(y)>\tau(y)$, there exists an $m_0$ so that for all $m\geq m_0$, $\tau(y)<T_\delta(x_m)$. Let $r$ be given as in (2) of  Assumption \ref{ass: ass_1}.  Pick an arbitrarily small $\delta_t\in(0, \inf_{m\geq m_0}T_\delta(x_{m})-\tau(y))$ so that for all $t\in(\tau(y), \tau(y)+\delta_t)$, 
\begin{equation*}
    \begin{split}
      &  \|\phi(t;x_m)\|_{\partial\Ds} \\ \leq &\norm{\phi(t;x_m)-\phi(\tau(y);x_m)} + \norm{\phi(\tau(y);x_m)-\phi(\tau(y);y)}\\
      \leq & \norm{f}(t-\tau(y)) +\eps_m  \\
      \leq & \norm{f}\delta_t +\eps_m
         \leq   r,
    \end{split}
\end{equation*}
where $\eps_m:=\|\phi(\tau(y);x_m)-\phi(\tau(y);y)\|\ra 0$. Then, by 2) of Assumption \ref{ass: ass_1}, $\omega(\phi(t; x_m))\gtrsim \psi(\|\phi(t;x_m)\|_{\partial\Ds})\gtrsim\psi(\norm{f}(t-\tau(y))+\eps_m)$ for all $t\in(\tau(y), \tau(y)+\delta_t)$. Then 
\begin{equation}
    \begin{split}
         V(x_m)  \geq & \int_{\tau(y)}^{\tau(y)+\delta_t}\omega(\phi(s;x_m))ds\\
        \gtrsim & \int_{\tau(y)}^{\tau(y)+\delta_t}\psi(\|f\|(s-\tau(y))+\eps_m)ds\\
        \gtrsim &  \int_{\eps_m}^{\eps_m+\norm{f}\delta_t}\psi(x)dx, 
    \end{split}
\end{equation}
and $\liminf_{m\ra\infty}V(x_m)\gtrsim \int_{0}^{\norm{f}\delta_t}\psi(x)dx$. Hence, $V(x_m)\ra\infty$ by (2) of Assumption \ref{ass: ass_1}.  

(3) and (4) proceed exactly the same as \cite[Proposition 1]{liu2025physics}  with the only exception that the working domain is $\D$ instead of $\Dd$\footnote{We also change the notion slightly to be consistent with the notation used in this paper.}. 
\end{proof}

The following example illustrates why item 2) of Assumption \ref{ass: ass_1} is needed to keep the values of $V(x)$ in the neighborhood of $\partial\Ds$ compatible with $q$.

\begin{ex}
    We show a counterexample where $\omega$ diverges near the boundary but $V(x_m)\not\to \infty$ for $x_m\to y\in\partial\D$. Consider $\dot{x} = -x$, and $T_\delta(x) = \inf\set{t\geq 0: |\phi(t;x)|<\delta}$. It can be verified that  $T_\delta(x) = \log\left(\frac{|x|}{\delta}\right)(1-\mathds{1}_{ [-\delta, \delta]}(x))$.
    
    Now let $\delta = 0.1$, $U = (-\infty, -1]\cup[1, \infty)$, $\omega(x) = -x^2(\log(1-x^2))$, and $x_m = 1-1/m$ for $m\geq 2$.  We also denote $t_m := T_{\delta}(x_m)$. Then, $\{x\in\Real: \tau(x) = \infty\} = (-1, 1)$ and $\D = \Real\cap (-1, 1) = (-1, 1)$ as well. Let 
    \begin{equation}
        I(x_m) =  \int_0^{t_m}\omega(\phi(s;x))ds. 
    \end{equation}
    Then, $\lim_{m\ra\infty}I(x_m) \approx 0.5$ and  $\phi(t_m;x_m) =   0.1$ for all $m\geq 2$. Therefore, $\tau(\phi(t_m;x_m)) = \infty$, $\phi(t_m;x_m)\in\D$,  
   and  $V(\phi(t_m;x_m)) = C$ ($\approx 2.5\times 10^{-5}$) for all $m\geq 2$.  Consequently,  $V(x_m) = I(x_m) + V(\phi(t_m;x_m))$, and $\lim_{m\ra\infty}V(x_m) <\infty$, whereas $V(1) = \infty$ by definition. 

   On the other hand, consider $h(x) = x^2$. Then $U^c = \set{x\in\Real: h(x)<1}$. Choosing $\omega(x) = \frac{x^2}{1-h(x)}$, one can verify that $\omega(x)\gtrsim 1/\norm{x}_U$ in the neighborhood of $\partial U$, and that $V(x_m) = \int_0^\infty\omega(\phi(t;x_m))dt = -\frac{\log(1-x_m^2)}{2}$, and $V(x_m)\ra\infty$ as $m\ra \infty$.  \qed
\end{ex}

Based on Theorem 1, and by similar reasoning as in \cite[Proposition. 2]{liu2025physics}, 
one can also show that $V$ in \eqref{E: V} is the unique continuous solution (in the viscosity sense) to 
\begin{equation}\label{E: lyap}
   - \nabla V(x)\cdot f(x)  -\omega(x) = 0, \;V(0) = 0, 
\end{equation}
in the general case, and that it is the unique $\mathcal{C}^1(\Omega)$ solution if $\omega\in\mathcal{C}^1(\Omega)$, where $\Omega\subset\D$ is any open set containing the origin. We omit the proof due to page limitation. 
\begin{rem}
    It is worth noting that $V$ itself is already a proper indicator for the origin on $\D$ by this construction. Moreover, one can find   class $\mathcal{K}_\infty$ functions $\underline{\alpha}$, $ \overline{\alpha}$, and $\alpha_3 = \underline{\alpha}^{-1}$,  such that $\underline{\alpha}(\omega(x))\leq V(x)\leq \overline{\alpha}(\omega(x))$. If $V$ is sufficiently smooth, then $\nabla V(x) \cdot f(x) = -\omega(x) \le -\alpha_3(V(x))$, which satisfies the conditions of Theorem~\ref{thm:conv1}. \qed
\end{rem}

\subsection{Lyapunov–barrier Function via Zubov Equation and Dirichlet Boundary Value Problem}
Although a maximal Lyapunov function $V$  on $\D$ can be obtained by solving \eqref{E: lyap}, it is unbounded and only well defined on $\D$. Therefore, solving it   requires working on a subregion of $\D$ to ensure $V$ remains bounded, which is typically unknown a priori. This  often yields a conservative estimate of $\D$ and a Lyapunov function valid only locally, owing to the aforementioned properties, thereby making the approach unsuitable for learning-based methods.  We are thus motivated to seek a transformation of 
$V$ that solves a corresponding PDE well-defined on an extended domain containing 
$\D$, so as to facilitate the learning process.

Following the idea of Zubov’s theorem, we apply a standard transformation of $V$ to obtain a scaled function $W$, such that the $\D = \{x\in\Real^n: W(x)<1\}$ and $\Real^n\setminus\D = \{x\in\Real^n: W(x)=1\}$.  The procedure follows \cite[Section 3.2]{liu2025physics}; however, we include it here for completeness. The transformation function $\beta: [0, \infty) \to \Real$ satisfies $\dot{\beta} = (1 - \beta)\varphi(\beta)$ with $\beta(0) = 0$, where $\varphi(s) > 0$ for all $s \ge 0$. Moreover, there exists a nonempty interval $I$ such that the function $G: I \to \Real$ defined by $s \mapsto (1 - s)\varphi(s)$ is monotonically decreasing on $I$.  It can be verified that $\beta$ is strictly increasing and that $\beta(s)\ra 1$ as $s\ra\infty$. 

\begin{rem}
    We commonly consider two straightforward choices of $\varphi(s) = \alpha$ and $\varphi(s) = \alpha(1+s)$ for some $\alpha>0$,  which result in  $\beta(s) = 1-e^{-\alpha s}$ and $\beta(s) = \tanh(\alpha s)$, respectively.\qed 
\end{rem}

With $V$ that satisfies \eqref{E: lyap}, let  \begin{equation}\label{E: W}
W(x) =
\begin{cases}
\displaystyle \beta(V(x)), & x\in\D, \\[1.2ex]
\displaystyle 1, & \text{otherwise}.
\end{cases}
\end{equation}
Then, by \cite[Chapter II, Proposition 2.5]{bardi1997optimal}, $W$  is the unique viscosity solution to  the Zubov equation
\begin{equation}\label{E: zubov_in}
   - \nabla W(x) \cdot f(x) + (1-W(x))\varphi(\beta(V(x))\omega(x) = 0  
\end{equation}
in $\D$, with $W(0) = 0$. With the properties of $V$ stated in Theorem \ref{thm:conv1} and the construction of $\beta$, $W(x)\ra 1$ as $x\ra\partial\D$. On the other hand, $W(x) = 1$ for all $x\in\Ds\setminus\D$ is a trivial solution to \eqref{E: zubov_in}.

With a boundary condition $g(x) = 1$ for all $x\in\partial\Ds$, we  can then work on the following Dirichlet boundary value problem for the Zubov equation  defined on $\Ds$, namely a \emph{Dirichlet-Zubov} formulation,
\begin{equation}\label{E: zubov}
\begin{cases}
\displaystyle - \nabla W(x) \cdot f(x) + (1-W(x))\varphi(\beta(V(x))\omega(x) = 0; \\[1.2ex]
\displaystyle W(0) = 0;  \;   W(x) = g(x), \;x\in\partial\Ds. 
\end{cases}
\end{equation}
Following a similar procedure as in     \cite[Theorem 2]{liu2025physics}, $W$ is the unique viscosity solution to \eqref{E: zubov}. 

\subsection{Equivalent Formulation of the Dirichlet–Zubov Problem} 

The solvability of \eqref{E: zubov} remains restricted to the domain 
$\Ds$ due to the definition of $\omega$. Moreover, the divergence rate of the integral \eqref{E: V} near the boundary layers of $\partial\Ds$ and 
$\partial\Dd$ may differ owing to $\omega$, resulting in a non-uniform convergence rate of $W(x)$ to $1$ as $x\ra\partial\D$. We are therefore further motivated to modify the equation to ensure its solvability on an extended domain.

\begin{prop}
  Let $\Omega\subseteq\Real^n$ strictly contain $\Ds$.   Suppose $\omega$ satisfies Assumption \ref{ass: ass_1} and is of the form given in  \eqref{E: omega_1}, i.e., $\omega = \theta(x)\frac{\|x\|^2}{(1-h(x))^k}$ with  $k\geq 1$ and $\theta\in \mathcal{C}_b^1(\Omega;\Real_+)$ as in Example 
\ref{eg: indicator_2}. Let 
\begin{equation}\label{E: tilde_f}
    \tilde{f}(x) = f(x)\cdot \frac{(1-h(x))^k}{\theta(x)}. 
\end{equation}
Then, $\D$ coincides with the domain of attraction $\tilde{\D}_0$ of the system $\dot{x} = \tilde{f}(x)$. 
\end{prop}

\begin{proof}
  For simplicity, let $\lambda(x) :=  \frac{(1-h(x))^k}{\theta(x)}$.    By Assumption \ref{ass: h}, $1-h(x)>0$ for all $x\in \Ds$, and $1-h(x)$ its maximum within $\Ds$ by the same reasoning as in Example~\ref{eg: indicator_2}. Therefore, $\lambda(x)>0$ for all $x\in \Ds$ and $\lambda(x) = 0$ for $x\in\partial\Ds$. By Assumption \ref{ass: h} and \ref{ass: ass_1}, the origin of $\dot{x} = \tilde{f}(x)$ is also asymptotic stable.  Denote $\psi(t;x_0)$ by the solution to $\dot{x} = \tilde{f}(x)$ with the initial condition $x(0) = x_0$.  
  
For all $x\in\tilde{\D}_0$, we have $\psi(t;x) \in \Ds$ for all $t\geq 0$.  Otherwise, there would exist some $T(x)\in(0, \infty)$ for some $x$ such that $\psi(T(x);x)\in\partial\Ds$ with $\psi(T(x);x)\neq 0$  and $\tilde{f}(\psi(t;x))=0$ for all $t\geq T(x)$. This implies $\psi(t;x)\neq 0$ for all $t\geq T(x)$,  which contradicts the asymptotic stability. Thus, $\lambda(\psi(t;x))>0$ for all $t\geq 0$ and all $x\in\tilde{\D}_0$.  For $s(t)$ such that 
  \begin{equation}
      \dot{s}(t) = \lambda(\psi(t);x)), \;s(0) = 0, 
  \end{equation}
$s(t)$ is strictly increasing.  Additionally, similar to the proof of 2) in Theorem~\ref{thm:conv1}, we can define $\tilde{T}_\delta(x) = \inf\set{t \ge 0 : \|\psi(t; x)\| < \delta} < \infty$ for a sufficiently small $\delta > 0$ such that $\B(0; \delta)$ is contained in a   closed, forward-invariant set $A$ with $A\cap U = \emptyset$. In this case, for all $t \ge \tilde{T}_\delta(x)$, we have $\psi(t; x) \in A$. Since $A$ is closed and $\Ds$ is open and bounded, there exists a $c > 0$ such that $\lambda(x) > c$ for all $x \in A$. It follows that $s(t) \geq \int_{\tilde{T}_\delta(x)}^t \lambda(\psi(s;x))ds\ra \infty$ as $t\ra\infty$. By the properties of $s$, the inverse $\sigma = s^{-1}$ is well defined. Moreover, for   $t = \sigma(s)$, we have $\sigma'(s) = 1 / \lambda(\psi(\sigma(s); x))$. Therefore,
$\dot{\psi}(\sigma(s); x) = \lambda(\psi(\sigma(s); x)) f(\psi(\sigma(s); x)) \sigma'(s) = f(\psi(\sigma(s); x))$.
By the uniqueness of the solution to the IVP, we have $\phi(s; x) = \psi(\sigma(s); x)$, and equivalently, $\psi(t; x) = \phi(s(t); x)$. Therefore, $\phi(s(t); x) \in \D$ by the properties of $\psi(t; x)$ and $s(t)$.

Conversely,  for all $x\in\D$, we define $\eta_x(s) = \lambda(\phi(s;x))$ for all $s\in[0, \infty)$. Then, $\eta_x(s)>0$ for all $s\geq 0$. One can verify that there exists a strictly increasing function $s_x(t)$, denoted by $s(t)$ for simplicity, which satisfies $\dot{s}(t) = \eta_x(s(t))$ with $s(0) = 0$. Note that, by Assumption~\ref{ass: h}, $\eta_x$ is uniformly bounded for all $x \in \D$. Hence, the scalar ODE for $s(t)$ has a linear growth bound, implying that $s(t)$ exists on $[0, \infty)$. Furthermore, by the uniform attraction property of $\phi(t; x)$, and with the same reasoning as above, we have $s(t) \to \infty$ as $t \to \infty$. Thus, again, we have $\psi(t; x) = \phi(s(t); x)$, and $\|\psi(t; x)\| \to 0$ as $t\ra\infty$.
\end{proof}

\begin{thm}\label{thm: modified}
     Let $\Omega\subseteq\Real^n$  be a compact region of interest (ROI) such that $\Ds\subseteq\Omega$.  Let $\tilde{f}$  be defined in \eqref{E: tilde_f} and $\psi(t;x_0)$ be the solution to $\dot x = \tilde{f}(x)$ with $x(0) = x_0$. Then, 
     \begin{enumerate}
         \item   $\tilde{V}(x) = \int_0^\infty\|\psi(t;x)\|^2dt$ is the the unique viscosity solution to $-\nabla \tilde{V}(x)\cdot \tilde{f}(x)-\norm{x}^2 = 0$ on $\D$ and $\tilde{V}(x) = V(x)$  for all $x\in\D$. 
         \item $\tilde{W}(x) = \beta(\tilde{V}(x))$ for $x\in\D$ and  $\tilde{W}(x) = 1$ elsewhere is the unique viscosity solution to
         \begin{equation}\label{E: zubov2}
\begin{cases}
\displaystyle   -\nabla\tilde{W}(x)\cdot \tilde{f}(x) + (1- \tilde{W}(x))\varphi(\beta(\tilde{V}(x))\norm{x}^2 = 0; \\[1.2ex]
\displaystyle \tilde{W}(0) = 0;  \;   \tilde{W}(x) = 1, \;x\in\partial\Omega. 
\end{cases}
\end{equation}
Furthermore $\tilde{W}(x) = W(x)$ for all $x\in\Ds$. 
     \end{enumerate}
\end{thm}
\begin{proof}
   The conclusion follows by multiplying the original PDEs by a positive factor, which yields exactly the PDEs for $\tilde{V}$ and $\tilde{W}$. Since multiplying a PDE by a strictly positive continuous function preserves viscosity sub- and supersolutions, both equations share the same set of solutions. With identical boundary data and uniqueness, the conclusion follows. 
\end{proof}

The modified \emph{Dirichlet–Zubov} problem \eqref{E: zubov2} ensures solvability over an extended ROI $\Omega$,  offering greater flexibility in the choice of domain compared with the original formulation.

\section{Numerical Examples}

In this section, we provide numerical examples to demonstrate the proposed approach. All numerical examples were conducted using a modified version of the Python toolbox LyZNet \cite{liu2024lyznet} for learning and verifying neural Lyapunov–barrier functions and the corresponding safe regions of attraction. All training and
verification tasks were performed on an Intel Xeon Gold 6326 CPU @ 2.90 GHz with 32 cores and 16 GB of RAM on a computing cluster. 

For all examples, we define obstacle sets a priori as $U_i := \{x\in\Real^n: h_i(x) \ge 1\}$ for some smooth functions $h_i$ and $i \ge 1$, where the $U:=\cup_i U_i$ is bounded. Let $\gamma(x) =1-h(x)$, where $h(x)=\max_i{h_i(x)}$. This is equivalent to specify $\Ds = \set{x\in\Real^n: h(x)<1}$. We consider $\omega(x)= \frac{\theta(x)\norm{x}^2}{\gamma(x)}$, where $\theta(x)$ is simply chosen as $1/\lambda$ for some parameter $\lambda>0$. This can be verified to be a proper indicator. In training the Lyapunov–barrier functions, we select a square ROI $\Omega\subseteq\Real^n$ and leverage Theorem~\ref{thm: modified} to train PINNs and solve  \eqref{E: zubov2} for $\tilde{W}$, where we set $\tilde{f}(x) = \lambda f(x)\gamma(x)$. We can then directly use LyZNet as if we were training the maximal Lyapunov function for the system $\dot x = \tilde{f}(x)$ as described in \cite{liu2025physics}. In the formal verification stage, the procedure follows exactly as in~\cite{liu2025physics}, except that we verify against the original vector field $f$, i.e., $\nabla \tilde{W} \cdot f$, instead of against $\tilde{f}$. For the examples below, we set $\lambda = 0.1$. Code for solving this and other examples is available at \url{https://git.uwaterloo.ca/hybrid-systems-lab/lyznet} under \texttt{examples/pinn-lbf}.

\begin{rem}
   It should be noted that alternative selections of $\gamma$ and $\theta$ are possible. For instance, $\gamma$ may be set to $\prod_i(1-h_i(x))$, while $\theta$ can be defined as a positive function of $x$. However, such $\theta(x)$ tends to be system-dependent, and an inappropriate choice may cause the trajectory evolution rate to vary considerably near obstacle boundaries and other regions, thereby complicating data generation. In contrast, the default choice presented above has shown robust performance across all tested case studies. \qed
\end{rem}

\begin{ex}[Reversed Van der Pol oscillator]
Consider the reversed Van der Pol oscillator
\begin{equation*}
\dot x_1 = -x_2, \quad 
\dot x_2 = x_1 - (1 - x_1^2)x_2, 
\end{equation*}
which has a stable equilibrium point at the origin. Consider $h_1(x) = 1 + \tfrac{1}{4} - \big((x_1 - 1)^2 + (x_2 - 1)^2\big) / 0.5^2$ and $h_2(x) = 1 + \tfrac{1}{4} - \big((x_1 + 1)^2 + (x_2 + 1)^2\big) / 0.5^2$, and sequentially add $U_i$ ($i=\{1, 2\}$) to the ROI $\Omega = [-2.5, 2.5] \times [-3.5, 3.5]$. By choosing $\beta(s) = \tanh(0.1s)$ and applying Theorem~\ref{thm: modified}, we visualize in Figure~\ref{fig:compare_qclf_data} the safe domain of attraction $\D$ for the single-obstacle ($h=h_1$) and two-obstacle ($h=\max_{i=1,2}h_i$) cases using $\{x: \tilde{W}(x)<1\}$, respectively. In contrast, Figure \ref{fig:verified-lbf-vpd} show the training results using PINNs and verification results by an SMT solver. Training for both cases took about 600 seconds. Verification took 700 seconds for the single-obstacle case and 300 seconds for the two-obstacle case.
\end{ex}

\begin{figure}[ht]
    \centering
    \includegraphics[scale = 0.38]{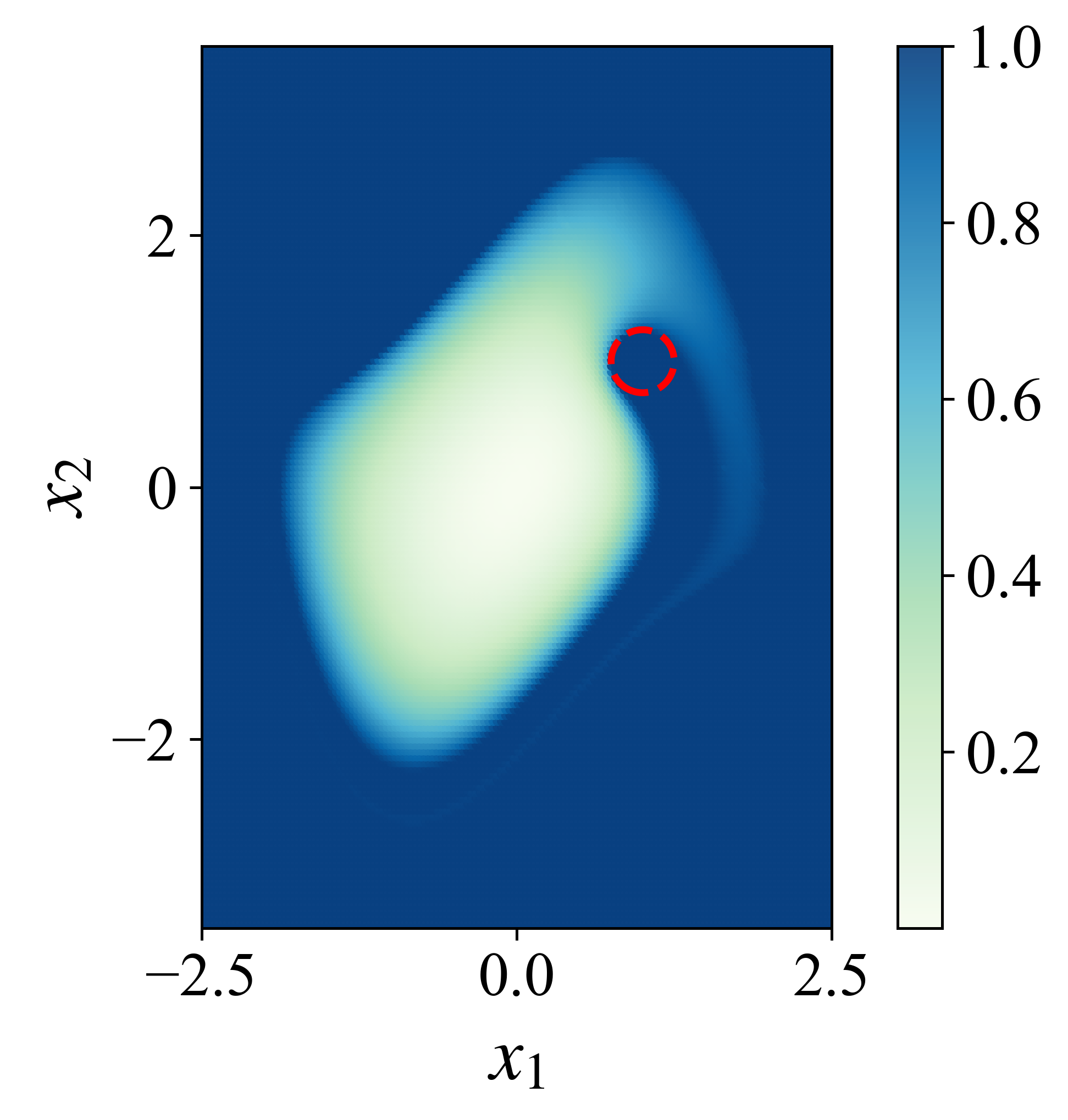}
    \includegraphics[scale = 0.38]{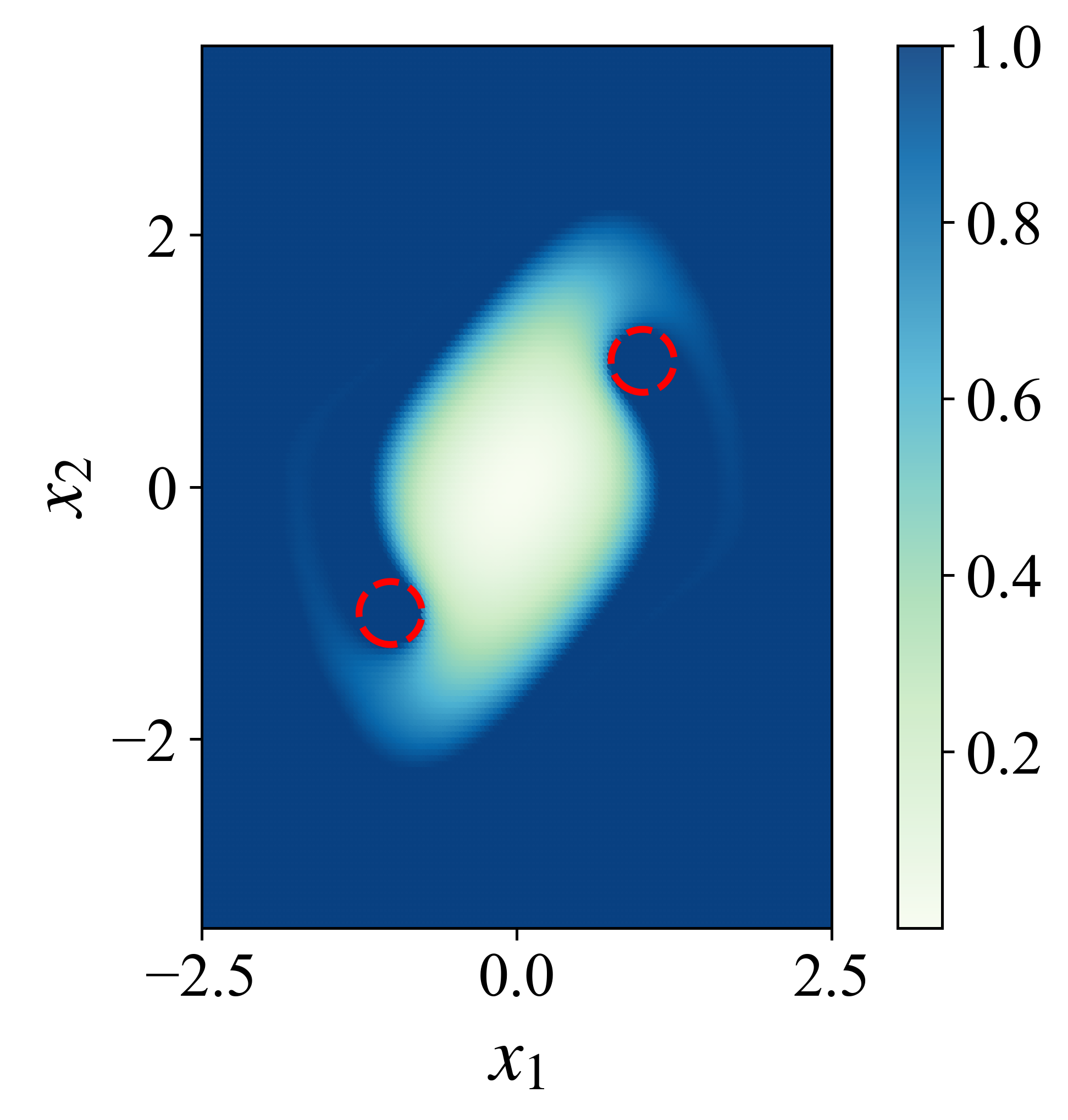}
    \caption{Visualization of the safe domain of attraction for the reversed Van der Pol oscillator with artificial obstacles. The regions enclosed by the  red curves represent the obstacles.}
    \label{fig:compare_qclf_data}
\end{figure}

\begin{figure}[ht]
    \centering
    \includegraphics[scale = 0.3]{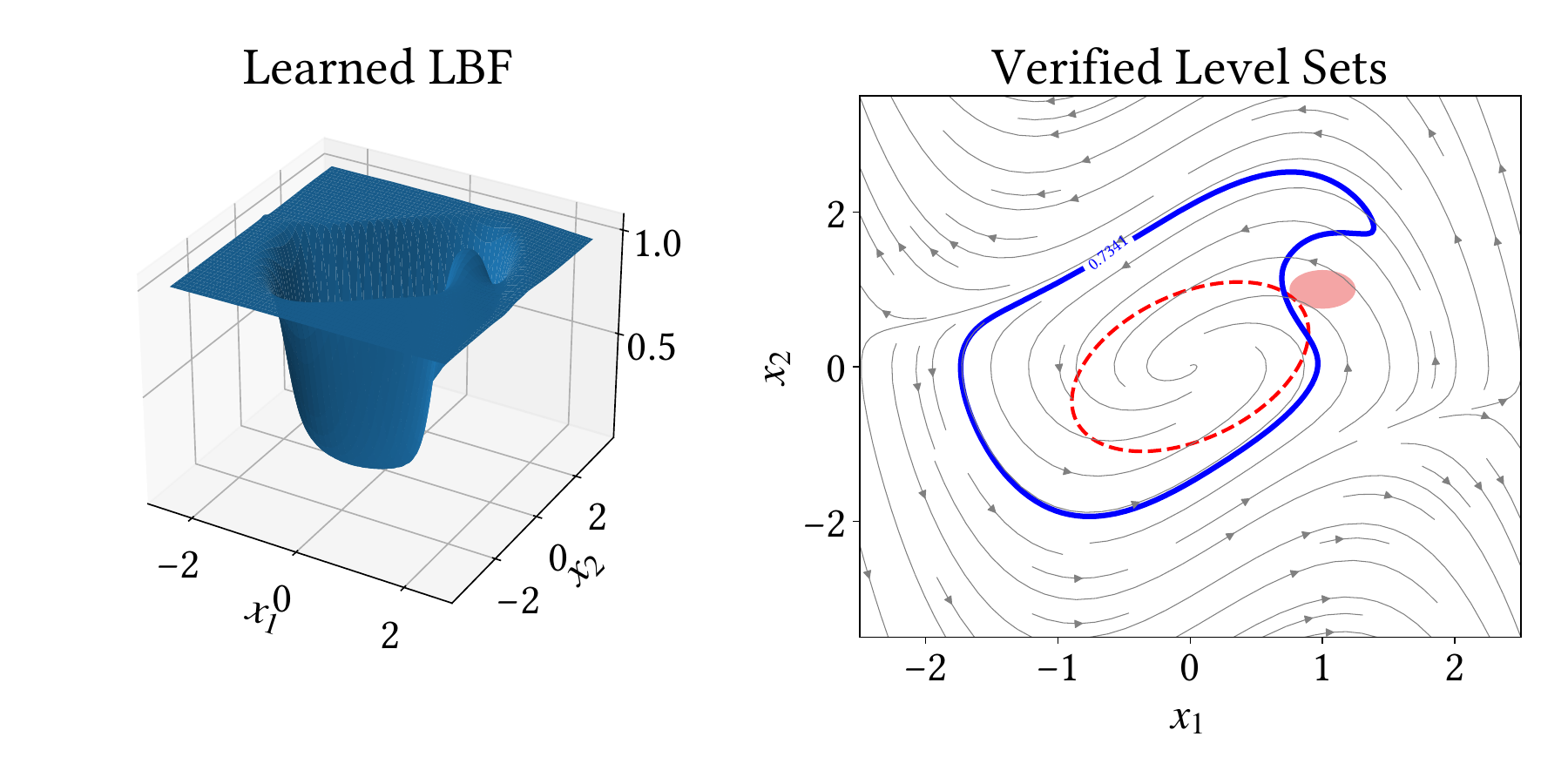}\\
    \includegraphics[scale = 0.3]{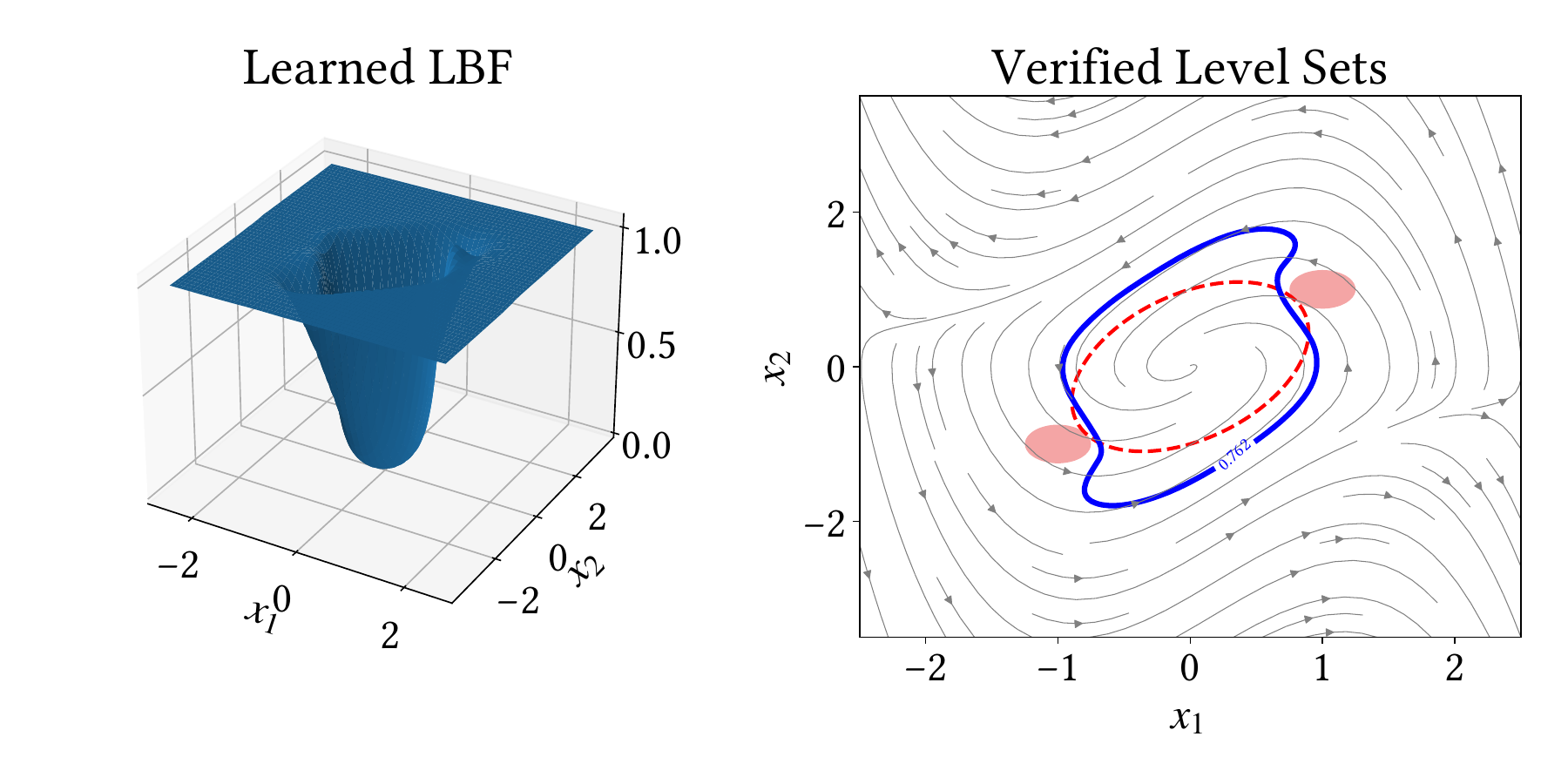}
    \caption{Learned Lyapunov–barrier functions and verified level sets (shown in solid blue) for the reversed Van der Pol oscillator with artificial obstacles. The dashed red curves represent the largest verifiable safe regions of attraction obtained using a quadratic Lyapunov function.}
    \label{fig:verified-lbf-vpd}
\end{figure}

The verified sublevel set, where the Lyapunov–barrier function condition holds, corresponds to a subset of $\D$. The verification using an SMT solver tends to be conservative, as it   returns only the largest level set within which the flow is completely forward invariant. Near the boundary layer of $\D$, the flow evolves slowly, causing the accumulated mass to persist longer and making the corresponding $W$ values less distinguishable from points satisfying $W(x) = 1$. The gradient of $W$ is also highly sensitive to numerical errors in this region, which explains the conservative estimation of the verified set. However, the verified level sets still demonstrate reasonably good learning performance, even in the presence of numerical errors near the boundary layer. 

\begin{figure}[ht]
    \centering
    \includegraphics[scale = 0.307]{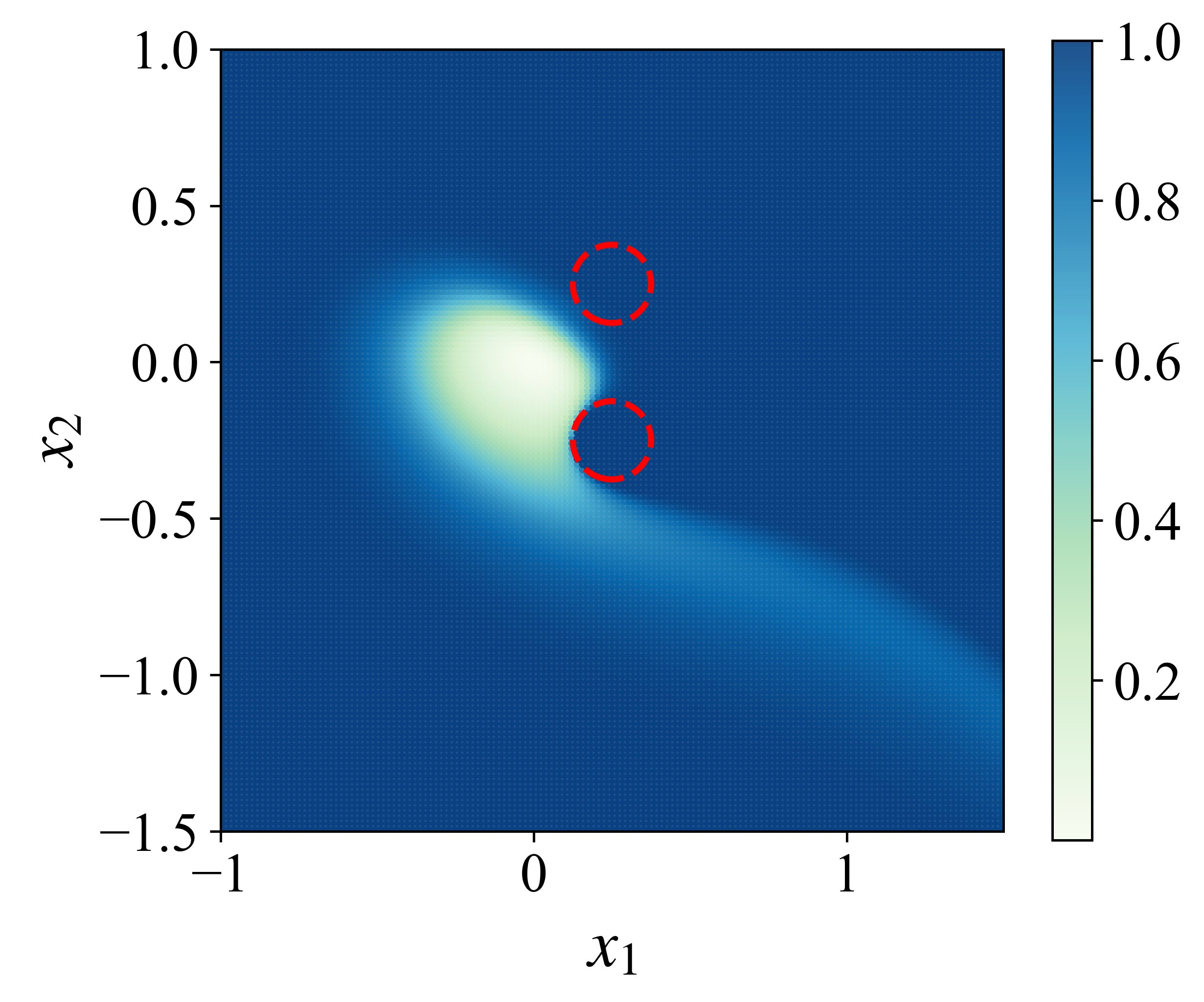}
    \includegraphics[scale = 0.307]{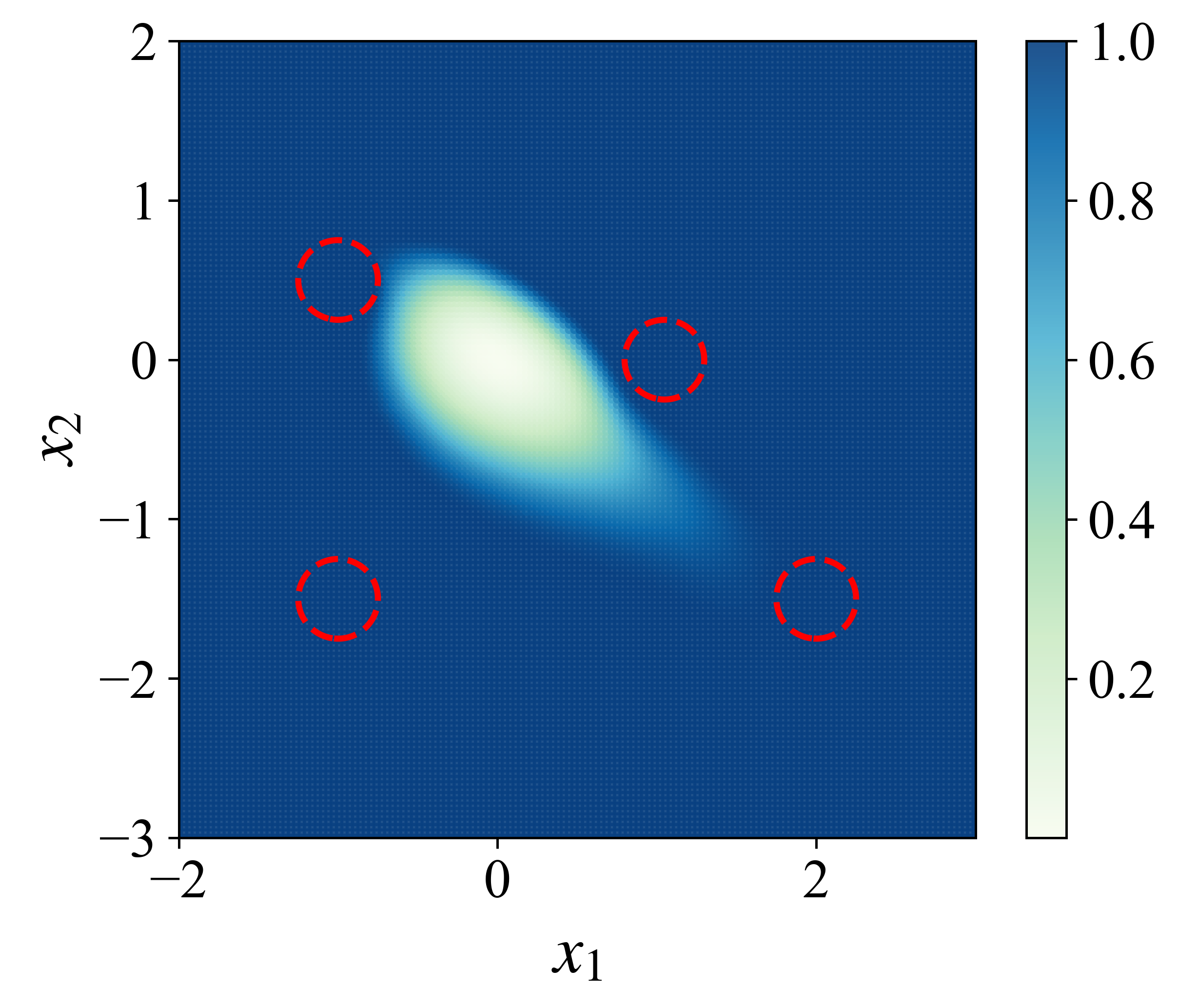}
    \caption{Visualization of the safe domain of attraction for the Two-Machine Power System with artificial obstacles. The regions enclosed by the  red curves represent the obstacles.}
    \label{fig:compare_qclf_data_2}
\end{figure}

\begin{ex}[Two-Machine Power System]
Consider the two-machine power system   modeled by
\begin{equation*}
    \dot{x}_1 = x_2, \quad \dot x_2 = -0.5x_2-(\sin(x_1+\delta)-\sin(\delta)),
\end{equation*}
\end{ex}
where $\delta = \pi/3$.  Note that the system has an unstable equilibrium point at $(\pi/3,0)$. We use this example to demonstrate that the proposed method remains effective when obstacles are placed sufficiently close to the asymptotically stable origin, and when multiple (three or more) obstacles are present. We list the functions that define the obstacles for each case as follows.
\begin{enumerate}
    \item $h_1(x) = 1 + \left(\frac{1}{8}\right)^2 - ((x_1 - 0.25)^2 + (x_2 - 0.25)^2)$;
    $h_2(x) = 1 + \left(\frac{1}{8}\right)^2 - ((x_1 - 0.25)^2 + (x_2 + 0.25)^2) $; $\Omega=[-1.0, 1.5] \times [-1.5, 1.0]$. 
    \item $h_1(x) = 1 + \left(\frac{1}{4}\right)^2 - ((x_1 - 2)^2 + (x_2 +1.5)^2)$;\\
    $h_2(x) = 1 + \left(\frac{1}{4}\right)^2 - ((x_1 +1)^2 + (x_2 -0.5)^2)$;\\
    $h_3(x) = 1 + \left(\frac{1}{4}\right)^2 - ((x_1 +1)^2 + (x_2 +1.5)^2)$;\\
    $h_4(x) = 1 + \left(\frac{1}{4}\right)^2 - ((x_1 -\delta)^2 + x_2^2)$;\\
    $\Omega = [-2.0, 3.0] \times [-3.0, 1.5]$. 
\end{enumerate}
We choose $\beta(s)=\tanh(0.025s)$ for case 1) and $\beta(s) = \tanh(0.01s)$ for case 2) and apply Theorem~\ref{thm: modified} to visualize the safe domain of attraction $\D$ for each case in Figure~\ref{fig:compare_qclf_data_2}. The training results obtained using PINNs and the verification results produced by the SMT solver are shown in Figure~\ref{fig:verified-lbf-vpd2}.

\begin{figure}[ht]
    \centering
    \includegraphics[scale = 0.3]{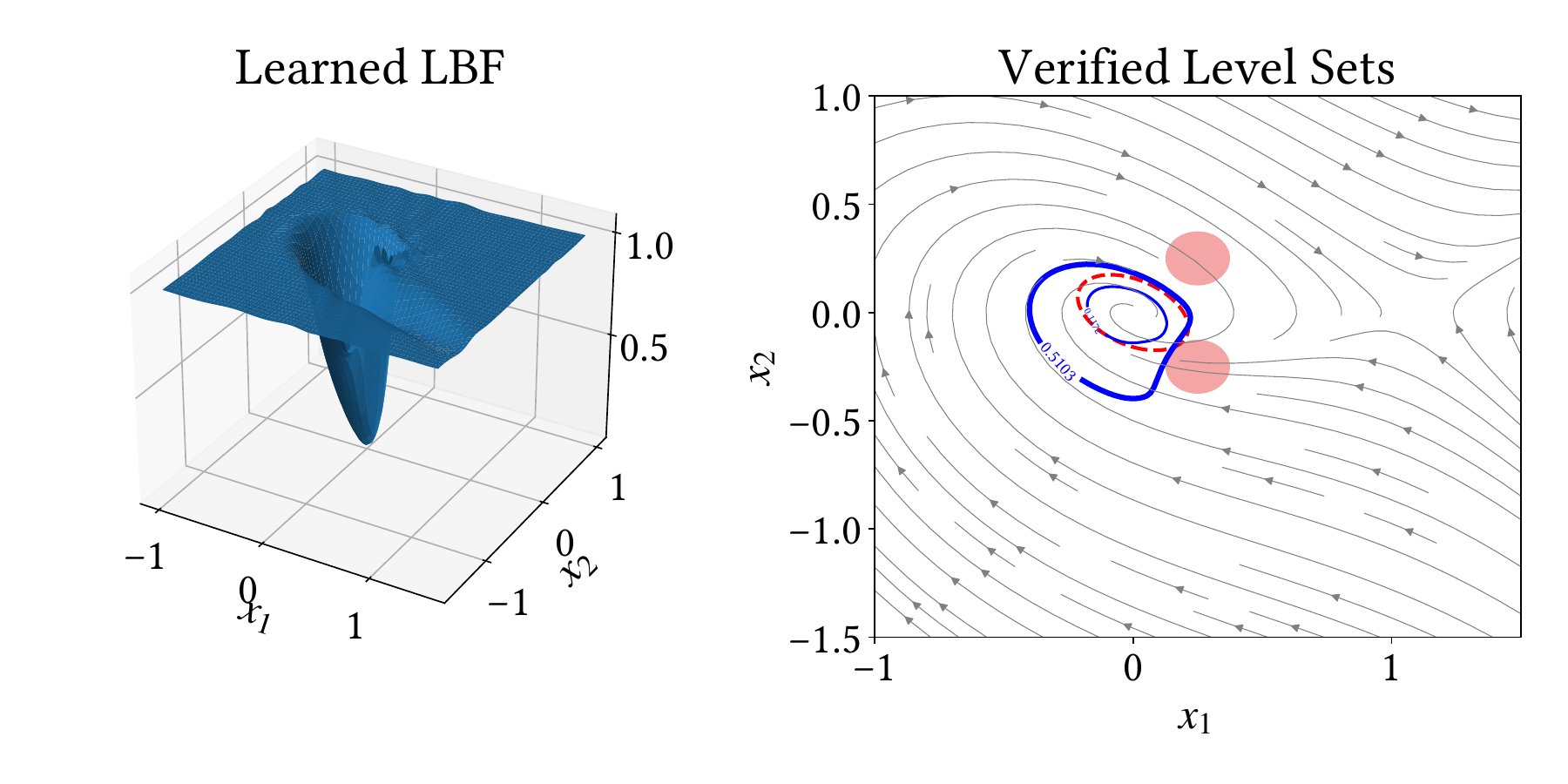}\\
    \includegraphics[scale = 0.3]{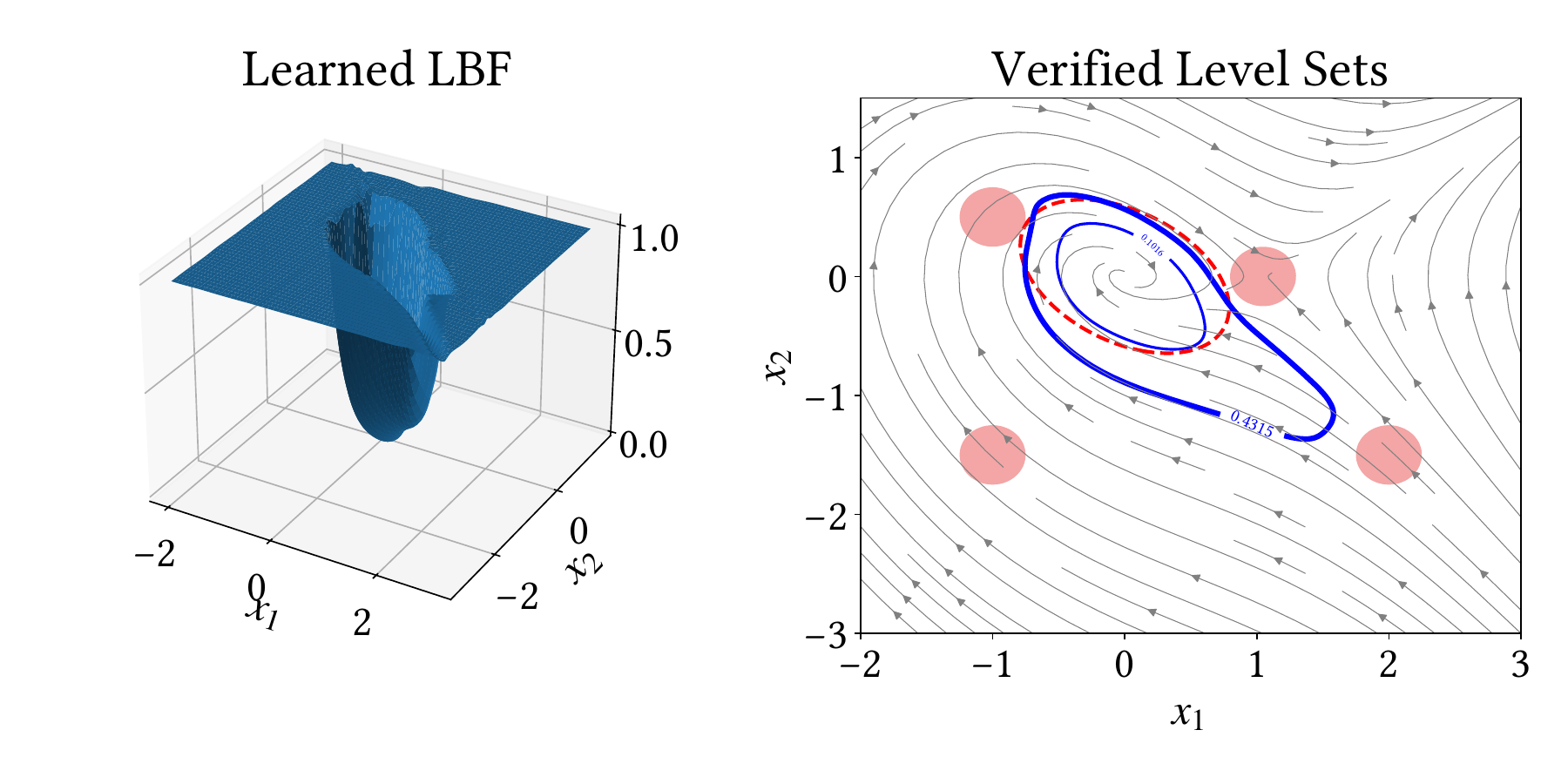}
    \caption{Learned Lyapunov–barrier functions and verified level sets (shown in solid blue) for the two-machine power system with artificial obstacles. The dashed red curves represent the largest verifiable safe regions of attraction obtained using a quadratic Lyapunov function.}
    \label{fig:verified-lbf-vpd2}
\end{figure}

The two examples demonstrate the effectiveness of the proposed method in recovering a reasonably large subset of $\D$, despite the challenge that the flow’s evolution rate varies abruptly near the boundaries of the obstacles. There also exists a very thin layer between the verified level set and the $\partial\Ds$, corresponding to a level set that cannot be verified by the SMT solver.

\section{Conclusion}

In this paper, we proposed a novel method for computing a unified Lyapunov–barrier function by solving a Dirichlet boundary problem of the Zubov equation, defined using a proper indicator function of the obstacles. We also introduced a modified formulation that ensures solvability over an extended region, making it compatible with the existing tool LyZNet. This approach enables the formal verification of the learned near-maximal Lyapunov–barrier function, ensuring that the derivative condition is strictly satisfied. The proposed framework was validated through multiple numerical examples, demonstrating its effectiveness and establishing a foundation for future research on learning CLBFs through PDE formulations. Another future direction is to integrate the current approach with \cite{meng2025learning} to learn a Lyapunov–barrier function for an unknown system.
\bibliographystyle{plain}
\bibliography{ecc26}

\end{document}